\documentclass[12pt,a4paper]{amsart}
\usepackage{mathabx}  

\usepackage[utf8]{inputenc}	
\usepackage[T1]{fontenc}
\usepackage[in,headings]{fullpage}
\usepackage{amsbsy, amscd, amsfonts, amsmath, amsrefs, amssymb, amsthm}

\usepackage{graphicx}

\usepackage{float}
\usepackage{color}   
\usepackage[colorlinks=true,linkcolor=blue,citecolor=blue,urlcolor=blue]{hyperref}
\usepackage{comment}
\excludecomment{A}

\newtheorem{theorem}{Theorem}

\theoremstyle{definition}

\newtheorem{remark}[theorem]{Remark}

\theoremstyle{remark}

\renewcommand{\Re}{\mathop{\mathrm{Re}}\nolimits}
\renewcommand{\Im}{\mathop{\mathrm{Im}}\nolimits}

\newcommand{\Li}{\mathop{\rm Li }\nolimits}
\newfont{\cmbsy}{cmbsy10}
\newfont{\cmmib}{cmmib10}

\begin{document}

\title{Riemann and the logarithmic derivatives of zeta}
\author[Arias de Reyna]{J. Arias de Reyna}
\address{%
Universidad de Sevilla \\ 
Facultad de Matem\'aticas \\ 
c/Tarfia, sn \\ 
41012-Sevilla \\ 
Spain.} 




\email{arias@us.es, ariasdereyna1947@gmail.com}


\begin{abstract}
In one of his posthumous papers, conserved in Göttingen, Riemann considers the derivatives of $\log\zeta(s)$ at the point $1/2$, giving explicit values for them.
Around 2010 we shared Riemann's value of the second derivative with some mathematicians. From that time I have been asked several times for references. So I decided to write this. Specially explaining the wonderful formulas
\[\frac{\zeta'(\frac12)}{\zeta(\frac12)}=\frac{\pi}{4}+\frac{\gamma}{2}+\frac{\log(8\pi)}{2},\quad
\frac{\zeta''(\frac12)}{\zeta(\frac12)}-\Bigl(\frac{\zeta'(\frac12)}{\zeta(\frac12)}\Bigr)^2=8-\frac{\pi^2}{4}-2G+2\sum_{n=1}^\infty\frac{1}{\alpha_n^2}\]
\end{abstract}

\maketitle

\section{Introduction}

At some moment I give notice of the two formulas 
\[\frac{\zeta'(\frac12)}{\zeta(\frac12)}=\frac{\pi}{4}+\frac{\gamma}{2}+\frac{\log(8\pi)}{2},\quad
\frac{\zeta''(\frac12)}{\zeta(\frac12)}-\Bigl(\frac{\zeta'(\frac12)}{\zeta(\frac12)}\Bigr)^2=8-\frac{\pi^2}{4}-2G+2\sum_{n=1}^\infty\frac{1}{\alpha_n^2}\]
to several mathematicians \cite{B}, \cite{G}. From time to time I received mails asking me for some reference. I could only say that they are due essentially to Riemann. In 2010 I get some not very good photocopies of some papers from Riemann. Later I get very good electronic copies of all Riemann papers on Number Theory in Göttingen. 

I learnt there that many formulas related to the zeta function, and usually attributed to later mathematicians were in fact known by Riemann himself. For example the Gram series
\[R(x)=\sum_{n=1}^\infty\frac{\mu(n)}{n}\Li(x^{1/n})=1+\sum_{k=1}^\infty\frac{(\log x)^k}{k!\,k\,\zeta(k+1)}.\]

It doesn't really matter who came up with a formula. When writing a draft—perhaps full of formulas that aren't entirely correct—one may realize that even if it isn't exactly the one we have before us, there is a relationship between a certain series and a certain constant or a certain integral. In a way, one has already seen the formula. Perhaps he will write it correctly later. Later, another mathematician writes it in a more elegant form. In this case, we can say that Voros \cite{V} has written similar equations more clearly. 

Does the fragment from Riemann in  Figure 1 convince you that the formulas in Theorem \ref{Riemann} are his? I think so, they were enough for me to write down the proof that Riemann had in mind. I hadn’t seen them anywhere else.  Later I found similar, not identical forms in Voros’s book \cite{V}.
In any case, it’s also possible that someone else has considered them and I’m simply unaware of it.

Among Riemann’s papers, a draft of a letter addressed to Weierstrass was found, reproduced in \cite{RN}. The letter was written after Riemann’s visit to Berlin, where he reported on his research into prime numbers. From the content of this letter, it can be inferred that it was Kronecker’s insistence that had prompted Riemann to speak about this research in Berlin. He informed Weierstrass that he was sending the paper to Berlin for publication in the Proceedings of the Berlin Academy.

\begin{figure}[H]
\begin{center}
\includegraphics[width=\hsize]{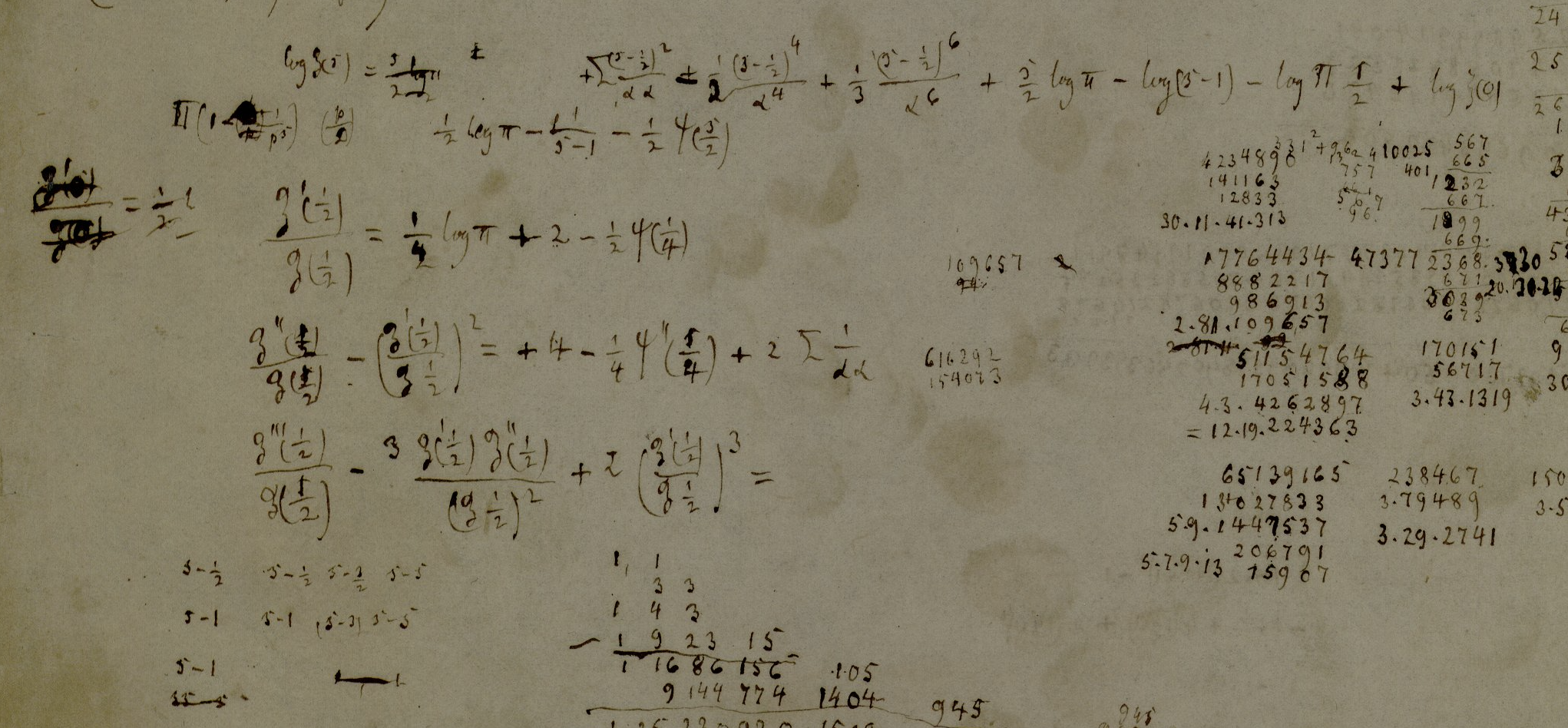}
\vspace{2mm}
\caption{
Fragment where Riemann writes his formulas\\about the logarithmic derivatives of $\zeta(s)$ at $s=\frac12$.}
\label{default}
\end{center}
\end{figure}

The Academy's proceedings published brief summaries of academic events, which explains why Riemann tried to keep his paper as short as possible. When reading the resulting article \cite{R}, one realizes that each sentence summarizes an entire paragraph.

He mentions in the letter that there are two statements whose proofs have not been fully completed. The first is that the number of zeros on the critical line $N_0(T)$ up to the point $T$ is asymptotically equivalent to the total number of zeros $N(T)$ (a statement that has not yet been proven).  The second is a step in the derivation of the equation for $\pi(x)$ that Landau \cite{L} later succeeded in proving.  

Riemann assumed that Weierstrass would have no trouble filling in the gaps in his exposition, except for the two mentioned. We know that this was not the case and that it took mathematicians decades to \emph{prove} Riemann’s claims. 

Riemann expressed doubts that the eight pages could be published in the Proceedings. Looking at it in hindsight, it may have been a pity that the paper was published. I say this because Riemann states that if he could not publish it, he would adapt it as a treatise for publication in Göttingen. And, certainly, Riemann could have written a book on the subject, given everything contained in the documents on number theory found in Göttingen.

\section{Riemann's derivation}

We usually define for $s$ and $t$ complex numbers
\begin{equation}
\xi(s)=\frac{s(s-1)}{2}\pi^{-s/2}\Gamma(s/2)\zeta(s),\qquad \Xi(t)=\xi(\tfrac12+it).
\end{equation} 
Then $\Xi(t)$ is an even entire function  with zeros $\alpha$ contained in the strip $|\Im(t)|<\frac12$. When $\alpha$ run through the zeros of $\Xi(t)$, the numbers $\frac12+i\alpha$ run through the complex zeros of $\zeta(s)$. The Riemann hypothesis is equivalent to say that the zeros $\alpha$ of $\Xi(t)$ are all real numbers. Under the Riemann hypothesis the numbers $\alpha$ are the ordinates of the zeros of zeta (and are denoted usually by the letter $\gamma$). We do not assume the Riemann hypothesis. 

For each zero $\alpha$ of $\Xi(t)$ the number $-\alpha$ is also a zero of $\Xi(t)$. We form a sequence $(\alpha_n)$ with the zeros of $\Xi(t)$ of positive real part,  ordered so that $0<\Re\alpha_1\le \Re\alpha_2\le \cdots$ repeating each zero according to its multiplicity. It is a fact that there is no zero of $\Xi(t)$ in the imaginary axis, so the numbers $\pm\,\alpha_n$ are all the zeros of the function $\Xi(t)$. 

All this is well known. Riemann derived the Hadamard product of $\Xi(t)$. His deduction has always been considered flawed, starting with Landau. It certainly is very concise, but in spite of Landau or the more detailed analysis of Edwards 
I have an unpublished paper in which I show how Riemann's suggestions can easily be developed into a rigorous proof (see what Edwards says in \cite{E}*{p~21}). 

In any case in most books I have consulted, instead of the product for $\Xi(t)$ the one for $\xi(s)$ is considered. So I stated Riemann's product formula.

\begin{theorem}
We have the expansion in product, valid for all complex numbers $t$
\begin{equation}\label{product}
\Xi(t)=\Xi(0)\prod_{n=1}^\infty\Bigl(1-\frac{t^2}{\alpha_n^2}\Bigr).
\end{equation}
\end{theorem}
\begin{proof}
The usual product for $\xi(s)$ is transformed in Edwards \cite{E}*{p.~31} in the product of Riemann. 
\end{proof}

Once we have the product for $\Xi(t)$ we may get a nice expression for $\zeta(s)$. This is trivial, but again is an expression not usually given explicitly.

\begin{theorem}
For all complex numbers $s$ (except the integers $\le 1$ we have 
\begin{equation}\label{zetaexp}
\zeta(s)=\frac{\pi^{s/2}}{2 (s-1)\Gamma(1+s/2)}\prod_{\Im\rho>0}\Bigl\{\Bigl(1-\frac{s}{\rho}\Bigr)\Bigl(1-\frac{s}{1-\rho}\Bigr)\Bigr\},
\end{equation}
where $\rho$ runs through the complex zeros of $\zeta(s)$ with imaginary part greater than $0$.
\end{theorem}
\begin{proof}
Let $s=\frac12+it$ where $s$ and $t$ are complex numbers, we will have $\Xi(t)=\xi(s)$. 
Then 
\[\Xi(t)=\Xi(0)\prod_{n=1}^\infty \Bigl(1-\frac{t^2}{\alpha_n^2}\Bigr)=\xi(s)=\frac{s(s-1)}{2}\pi^{-s/2}\Gamma(s/2)\zeta(s).\]
If $s$ is not an integer, we have 
\begin{equation}\label{partial}
\zeta(s)=\Xi(0)\frac{\pi^{s/2}}{(s-1)\Gamma(1+s/2)}\prod_{n=1}^\infty\Bigl(1+\frac{(s-1/2)^2}{\alpha_n^2}\Bigr).
\end{equation}
The quadratic function can be factored 
\[\Bigl(1+\frac{(s-1/2)^2}{\alpha_n^2}\Bigr)=\Bigl(1+\frac{1}{4\alpha_n^2}\Bigr)\Bigl(1-\frac{s}{\frac12+i\alpha_n}\Bigr)\Bigl(1-\frac{s}{\frac12-i\alpha_n}\Bigr).\]
Also note that 
\[\Xi(0)\prod_{n=1}^\infty\Bigl(1+\frac{1}{4\alpha_n^2}\Bigr)=\Xi(i/2)=\xi(0)=-\zeta(0)=\frac12.\]
This yields \eqref{zetaexp} after noticing that $\frac12+i\alpha_n=\rho_n$ run through the complex zeros of zeta with imaginary part $>0$. 
\end{proof}

\begin{theorem}\label{Riemann}
We have the particular values
\begin{equation}
\frac{\zeta'(\frac12)}{\zeta(\frac12)}=\frac{\pi}{4}+\frac{\gamma}{2}+\frac{\log(8\pi)}{2},\quad
\frac{\zeta''(\frac12)}{\zeta(\frac12)}-\Bigl(\frac{\zeta'(\frac12)}{\zeta(\frac12)}\Bigr)^2=8-\frac{\pi^2}{4}-2G+2\sum_{n=1}^\infty\frac{1}{\alpha_n^2}.
\end{equation}
where $\gamma=0.577\dots  $ is Euler constant and $G=0.915\dots$ is Catalan's constant. 
\end{theorem}
\begin{proof}
For $|s-\frac12|<|\alpha_1|$ we have $|s-\frac12|<|\alpha_n|$ for all $n$. It follows that 
\[\log\prod_{n=1}^\infty\Bigl(1+\frac{(s-1/2)^2}{\alpha_n^2}\Bigr)=\sum_{n=1}^\infty
\log\Bigl(1+\frac{(s-1/2)^2}{\alpha_n^2}\Bigr)=-\sum_{n=1}^\infty\sum_{k=1}^\infty 
(-1)^k\frac{(s-1/2)^{2k}}{k\alpha_n^{2k}}\]
where the double series, which converges absolutely in the disc $|s-\frac12|<|\alpha_1|$ represents a holomorphic determination of the logarithm in the disc.

Therefore,
by logarithmic differentiation of \eqref{partial}, putting $\Gamma(1+s/2)=(s/2)\Gamma(s/2)$ 
\[\frac{\zeta'(s)}{\zeta(s)}=\frac{1}{2}\log\pi-\frac{1}{s-1}-\frac1s-\frac{\Gamma'(s/2)}{2\Gamma(s/2)}-2\sum_{n=1}^\infty \sum_{k=1}^\infty (-1)^k\frac{(s-\frac12)^{2k-1}}{\alpha_n^{2k}}.\]
Putting here $s=\frac12$, the sum vanishes, and we get
\[\frac{\zeta'(\frac12)}{\zeta(\frac12)}=\frac{\log \pi}{2}-\frac{\Gamma'(1/4)}{2\Gamma(1/4)}=\frac{\log \pi}{2}-\frac12\psi(1/4),\]
Where we used the notation $\psi(s)=\Gamma'(s)/\Gamma(s)$. 

It is well known that 
\begin{equation}\label{part}
\psi(1/4)=-\frac{\pi}{2}-\gamma-3\log 2,\qquad \psi'(1/4)=\pi^2+8G,
\end{equation}
where $\gamma=0.577\dots  $ is Euler constant and $G$ is Catalan's constant. 

Substituting we get 
\[\frac{\zeta'(\frac12)}{\zeta(\frac12)}=\frac{\pi}{4}+\frac{\gamma}{2}+\frac{\log(8\pi)}{2}.\]
Differentiating again the equation (the series converge absolutely)
\[\frac{\zeta'(s)}{\zeta(s)}=\frac{1}{2}\log\pi-\frac{1}{s-1}-\frac1s-\frac12\psi(s/2)-2\sum_{n=1}^\infty \sum_{k=1}^\infty (-1)^k\frac{(s-\frac12)^{2k-1}}{\alpha_n^{2k}}\] yields
\[\frac{d}{ds}\Bigl(\frac{\zeta'(s)}{\zeta(s)}\Bigr)=\frac{1}{(s-1)^2}+\frac{1}{s^2}-\frac14\psi'(s/2)-2\sum_{n=1}^\infty \sum_{k=1}^\infty (-1)^k(2k-1)\frac{(s-\frac12)^{2k-2}}{\alpha_n^{2k}}\]
Here we substitute $s=\frac12$
\[\Bigl.\frac{d}{ds}\Bigl(\frac{\zeta'(s)}{\zeta(s)}\Bigr)\Bigr|_{s=\frac12}=8-\frac14\psi'(1/4)+2\sum_{n=1}^\infty\frac{1}{\alpha_n^2}.\]
That is 
\[\frac{\zeta''(\frac12)}{\zeta(\frac12)}-\Bigl(\frac{\zeta'(\frac12)}{\zeta(\frac12)}\Bigr)^2=8-\frac{\pi^2}{4}-2G+2\sum_{n=1}^\infty\frac{1}{\alpha_n^2}.\qedhere\]
\end{proof}

\begin{remark}
Riemann write the next derivative, we see that the process can continue. That there is a difference between the even and odd derivatives.  But forget about continue this to see for that the approach of Voros. 
\end{remark}

\begin{remark}
The first equation in \eqref{part} is well known, it is a particular case of an expression for $\psi(p/q)$ due to Gauss. 

The second equation in \eqref{part} is also well known. For example it appears in \cite{F}*{p.55}. It is easy to derive from Mittag-Leffler expansion of $\psi'(z)$, the definition of $G=\sum (-1)^n(2n+1)^2$, 
and the series of Euler  $\sum(2n+1)^{-2}$.  
\end{remark}

\section{Voros equations}

We must cite Voros \cite{V} for this material. But what I write is not directly there. 

Denote by $\chi$ the non trivial Dirichlet character modulo 4, then 
\[L(s,\chi)=\sum_{k=0}^\infty \frac{(-1)^k}{(2k+1)^s},\quad \Bigl(1-\frac{1}{2^s}\Bigr)\zeta(s)=\sum_{k=0}^\infty\frac{1}{(2k+1)^s}.\]
We have 
\[L(2n+1,\chi)=\frac{(\pi/2)^{2n+1}}{2(2n)!}|E_{2n}|, \quad \zeta(2n)=\frac{(2\pi)^{2n}}{2(2n)!}|B_{2n}|,\qquad n\ge1\]
where $E_n$ and $B_n$ are Euler and Bernoulli numbers.

\begin{theorem}
For $n\ge 3$ odd
\begin{equation}
\Bigl.\frac{1}{(n-1)! 2^{n-1}}\frac{d^{n-1}}{ds^{n-1}}\frac{\zeta'(s)}{\zeta(s)}\Bigr|_{s=\frac12}=\Bigl(1-\frac{1}{2^n}\Bigr)\zeta(n)+L(n,\chi).
\end{equation}
For $n\ge 2$ even, 
\begin{equation}
\Bigl.\frac{1}{(n-1)! 2^{n-1}}\frac{d^{n-1}}{ds^{n-1}}\frac{\zeta'(s)}{\zeta(s)}\Bigr|_{s=\frac12}=4-\frac{(-1)^{n/2}}{2^{n}}\sum_{\Re\alpha>0}\frac{4}{\alpha^n}-L(n,\chi)-\Bigl(1-\frac{1}{2^n}\Bigr)\zeta(n).
\end{equation}
\end{theorem}
\begin{proof}
We start with 
\[\frac{\zeta'(s)}{\zeta(s)}=-\frac{1}{s-1}+\sum_\rho\Bigl(\frac{1}{s-\rho}-\frac{1}{\rho}\Bigr)
+\sum_{k=1}^\infty \Bigl(\frac{1}{s+2k}-\frac{1}{2k}\Bigr)+B\]
Differentiating this 
\[\frac{d^{n-1}}{ds^{n-1}}\frac{\zeta'(s)}{\zeta(s)}=\frac{(-1)^{n} (n-1)!}{(s-1)^{n}}-\sum_{\rho}\frac{(-1)^n (n-1)!}{(s-\rho)^{n}}-\sum_{k=1}^\infty \frac{(-1)^n (n-1)!}{(s+2k)^{n}}.\]
Or also 
\[\frac{d^{n-1}}{ds^{n-1}}\frac{\zeta'(s)}{\zeta(s)}=\frac{ (n-1)!}{(1-s)^{n}}-\sum_{\rho}\frac{(n-1)!}{(\rho-s)^{n}}-\sum_{k=1}^\infty \frac{(-1)^n (n-1)!}{(s+2k)^{n}}.\]
Now take $s=\frac12$
\[\Bigl.\frac{d^{n-1}}{ds^{n-1}}\frac{\zeta'(s)}{\zeta(s)}\Bigr|_{s=\frac12}=2^n (n-1)!-\sum_{\alpha}\frac{(n-1)!}{(i\alpha)^{n}}-\sum_{k=1}^\infty \frac{(-1)^n (n-1)!}{(\frac12+2k)^{n}}.\]
We separate the cases in which $n$ is odd or even. When $n$ is odd, the sum in $\alpha$ vanishes since the term with $(i\alpha)^n$ is opposite to that with $(-i\alpha)^n$, then we have, for $n$ odd
\begin{align*}
\Bigl.\frac{d^{n-1}}{ds^{n-1}}\frac{\zeta'(s)}{\zeta(s)}\Bigr|_{s=\frac12}&=2^n (n-1)!+\sum_{k=1}^\infty \frac{(n-1)!}{(\frac12+2k)^{n}}=2^n (n-1)!\sum_{k=0}^\infty \frac{1}{(4k+1)^n}\\
&=2^{n-1} (n-1)!\Bigl\{L(n,\chi)+\Bigl(1-\frac{1}{2^n}\Bigr)\zeta(n)\Bigr\}.
\end{align*}
When $n$ is even we obtain 
\begin{align*}
\Bigl.\frac{d^{n-1}}{ds^{n-1}}\frac{\zeta'(s)}{\zeta(s)}\Bigr|_{s=\frac12}&=2^n (n-1)!-(-1)^{n/2}\sum_{\Re\alpha>0}\frac{2 (n-1)!}{\alpha^{n}}-\sum_{k=1}^\infty \frac{ (n-1)!}{(\frac12+2k)^{n}}\\
&=2^n(n-1)!\Bigl(2-\frac{(-1)^{n/2}}{2^{n}}\sum_{\Re\alpha>0}\frac{2}{\alpha^n}-\sum_{k=0}^\infty\frac{1}{(4k+1)^n}\Bigr)\\
&=2^{n-1}(n-1)!\Bigl(4-\frac{(-1)^{n/2}}{2^{n}}\sum_{\Re\alpha>0}\frac{4}{\alpha^n}-L(n,\chi)-\Bigl(1-\frac{1}{2^n}\Bigr)\zeta(n)\Bigr)\qedhere
\end{align*}
\end{proof}

We may write some of the equations we get:
\[
\begin{aligned}
&\frac{\zeta'(\frac12)}{\zeta(\frac12)}=\frac{\pi}{4}+\frac{\gamma}{2}+\frac{\log(8\pi)}{2},\\
&\frac{\zeta''(\frac12)}{\zeta(\frac12)}-\Bigl(\frac{\zeta'(\frac12)}{\zeta(\frac12)}\Bigr)^2=8-\frac{\pi^2}{4}-2G+2\sum_{n=1}^\infty\frac{1}{\alpha_n^2},\\
&\frac{\zeta'''(\frac12)}{\zeta(\frac12)}-3\frac{\zeta''(\frac12)\zeta'(\frac12)}
{\zeta(\frac12)^2}+2\frac{\zeta'(\frac12)^3}
{\zeta(\frac12)^3}=\frac{\pi^3}{4}+7\zeta(3),\\
&\frac{\zeta^{(4)}(\frac12)}{\zeta(\frac12)}
-4\frac{\zeta'(\frac12)\zeta'''(\frac12)}{\zeta(\frac12)^2}
-3\frac{\zeta''(\frac12)^2}{\zeta(\frac12)^2}
+12\frac{\zeta'(\frac12)^2\zeta''(\frac12)}{\zeta(\frac12)^3}
-6\frac{\zeta'(\frac12)^4}{\zeta(\frac12)^4}\\&\qquad=192-\frac{\pi^4}{2}-48L(4,\chi)-12\sum_{n=1}^\infty\frac{1}{\alpha_n^4},\\
&\frac{\zeta^{(5)}(\frac12)}{\zeta(\frac12)}
-5\frac{\zeta'(\frac12)\zeta^{(4)}(\frac12)}{\zeta(\frac12)^2}
-10\frac{\zeta''(\frac12)\zeta'''(\frac12)}{\zeta(\frac12)^2}
+20\frac{\zeta'(\frac12)^2\zeta'''(\frac12)}{\zeta(\frac12)^3}
+30\frac{\zeta'(\frac12)\zeta''(\frac12)^2}{\zeta(\frac12)^3}\\
&\quad-60\frac{\zeta'(\frac12)^3\zeta''(\frac12)}{\zeta(\frac12)^4}
+24\frac{\zeta'(\frac12)^5}{\zeta(\frac12)^5}
=12\Bigl(31\zeta(5)+\frac{5\pi^5}{48}\Bigr).
\end{aligned}
\]

We have checked numerically these formulas to 40 digits, using mpmath \cite{mpmath} and specially the implementation of the secondary zeta function \cite{A}. 



\end{document}